\documentclass[11pt]{amsart} 
\usepackage[foot]{amsaddr}
\usepackage{amsmath,amssymb,bm,color,float} 
\usepackage{graphicx}
\usepackage{mathrsfs}
\usepackage{dsfont}
\usepackage{pifont}
\usepackage{wrapfig}
\usepackage{hyperref}
\usepackage{multirow}

\def\3bar{{|\hspace{-.02in}|\hspace{-.02in}|}}

\newtheorem{theorem}{Theorem}[section]
\newtheorem{lemma}[theorem]{Lemma}

\def\t#1{\operatorname{#1}}

\numberwithin{equation}{section}

\def\an#1{\begin{align}#1\end{align}}
 \def\t#1{\hbox{\rm{#1}}}

  \numberwithin{equation}{section}

\def\3bar{{|\hspace{-.02in}|\hspace{-.02in}|}}

\title[Least-Squares Weak Galerkin]{A Least-Squares Weak Galerkin Finite Element Scheme for Cauchy Problems in Helmholtz}

\author{Chunmei Wang} 
\address{Department of Mathematics, University of Florida, Gainesville, FL 32611, USA.} 
\email{chunmei.wang@ufl.edu} 

\author{Shangyou Zhang}
\address{Department of Mathematical Sciences, University of Delaware, Newark, DE 19716, USA} 
\email{szhang@udel.edu}

\begin{document}
\begin{abstract}
This paper introduces and rigorously analyzes a least-squares weak Galerkin (LS-WG) finite element method for the severely ill-posed Cauchy problem associated with the Helmholtz equation. By utilizing a weak Laplacian operator defined on a space of discontinuous functions, the proposed framework facilitates the seamless treatment of complex boundary conditions and internal interfaces. We emphasize the geometric flexibility of the LS-WG scheme on general polygonal and polyhedral partitions. Furthermore, we prove the uniqueness of the numerical solution and derive optimal-order error estimates with respect to a specifically designed discrete energy norm. Extensive numerical experiments validate the theoretical convergence rates and demonstrate the algorithm's robustness and efficiency over traditional Galerkin approaches.
\end{abstract}

 
  \maketitle 
\section{Introduction}

In this work, we introduce a least-squares weak Galerkin (LS-WG) finite element method based on completely discontinuous functions to solve the Cauchy problem for the Helmholtz equation. To present the proposed method, we consider the following model problem: find a function $u$ such that
\begin{equation}\label{model}
\begin{cases} 
 \Delta u + k^2 u = f & \text{in } \Omega, \\
u = g_1, \quad \nabla u \cdot \mathbf{n} = g_2 & \text{on } \Gamma_1,
\end{cases}
\end{equation}
where $k > 0$ represents the wavenumber and $\Omega \subset \mathbb{R}^d$ ($d=2, 3$) is an open, bounded, polytopal domain. The boundary $\partial\Omega$ is partitioned into two disjoint, relatively open subsets $\Gamma_1$ and $\Gamma_2$, such that $\partial\Omega = \overline{\Gamma}_1 \cup \overline{\Gamma}_2$. Here, $\Gamma_1$ (with strictly positive measure, $|\Gamma_1|>0$) represents the accessible boundary where both Dirichlet and Neumann data are prescribed, while $\Gamma_2$ is the inaccessible boundary where no data is provided. 

The vector $\mathbf{n}$ denotes the unit outward normal to $\partial\Omega$. The source term $f$ and the boundary data $g_1, g_2$ are assumed to be sufficiently regular to ensure the well-posedness of the corresponding discrete formulation. The objective of this study is to reconstruct the solution $u$ throughout the domain $\Omega$ and its trace on the unknown boundary $\Gamma_2$. Given the inherent ill-posedness of this boundary value problem, we employ a stabilized numerical approach through the LS-WG framework to ensure theoretical accuracy and computational robustness.

The Cauchy problem for the Helmholtz equation, which seeks to determine an interior solution from over-determined data on a subset of the boundary, is a fundamental challenge in mathematical physics due to its severe ill-posedness in the sense of Hadamard \cite{hadamard}. Despite the inherent instability---where infinitesimal perturbations in boundary data can lead to exponentially large errors in the solution---the problem is of paramount importance in various inverse scattering and imaging applications. It arises naturally in non-destructive evaluation for detecting internal flaws or material properties from surface measurements, and in acoustic holography for reconstructing noise sources \cite{isakov, colton2013}. Furthermore, it plays a critical role in medical imaging, such as thermoacoustic tomography, and in geophysical exploration for subsurface material identification \cite{tadi, berntsson}. The ability to accurately reconstruct wave fields from partial boundary data is essential for interpreting physical phenomena where the interior of the domain is inaccessible for direct measurement.

Due to this inherent instability, classical numerical methods often fail, necessitating specialized regularization techniques. A well-known approach is the method of quasi-reversibility, which replaces the ill-posed second-order operator with a higher-order perturbed operator to restore stability \cite{lattes1969}. Iterative methods, such as the alternating procedure proposed by Kozlov et al., have also been widely adopted, wherein a sequence of well-posed boundary value problems is solved to converge toward the Cauchy solution \cite{kozlov1991}. In recent years, variational and finite element-based frameworks have gained traction. Specifically, the Least-Squares Weak Galerkin (LS-WG) method has emerged as a robust alternative; it transforms the ill-posed Cauchy problem into a symmetric positive definite (SPD) system while providing significant geometric flexibility on polygonal meshes \cite{wangzhang_cauchy}. Other approaches include boundary element methods (BEM) combined with Tikhonov regularization, Landweber iterations, and mollification methods that employ filtering to stabilize high-frequency components \cite{hanke, nanfuka}.

The Weak Galerkin (WG) finite element method, originally introduced in \cite{ellip_JCAM2013} and extensively expanded in \cite{wg1, wg2, wg3, wg4, wg5, wg6, wg7, wg8, wg9, wg10, wg11, wg12, wg13, wg14, wg15, wg16, wg17, wg18, wg19, wg20, wg21, itera, wz2023, wy3655}, represents a significant departure from traditional continuous finite element methods. By utilizing weak derivatives and enforcing weak continuity across element interfaces via specifically designed stabilizers, the WG framework naturally accommodates general polygonal and polyhedral meshes. A notable evolution of this framework is the Primal-Dual Weak Galerkin (PDWG) method \cite{pdwg1, pdwg2, pdwg3, pdwg4, pdwg5, pdwg6, pdwg7, pdwg8, pdwg9, pdwg10, pdwg11, pdwg12, pdwg13, pdwg14, pdwg15}, which formulates numerical approximations as constrained minimization problems. While PDWG offers favorable stability for non-self-adjoint problems, such as linear transport \cite{wwhyperbolic}, it often increases the number of global unknowns through the introduction of dual variables.

The proposed LS-WG method offers several transformative advantages for the Helmholtz Cauchy problem. While traditional finite element discretizations of the Helmholtz equation typically result in indefinite linear systems---which grow increasingly difficult to solve as the wavenumber $k$ increases---the least-squares formulation yields an \textbf{SPD} discrete system. This property is a significant computational asset, as it allows for the use of high-performance iterative solvers like the Conjugate Gradient (CG) method. By combining the inherent stability of least-squares minimization with the versatile nature of weak derivatives, the LS-WG approach provides a resilient and scalable numerical platform for recovering solutions in ill-posed settings.

The primary contributions of this paper are summarized as follows:
\begin{itemize}
    \item \textbf{Symmetry and Positivity:} The resulting discretized bilinear form is inherently symmetric and positive definite, yielding an SPD linear system that facilitates efficient computation via the CG method.
    \item \textbf{Stability at High Wavenumbers:} The least-squares framework provides a natural stabilization mechanism that effectively mitigates the numerical instabilities often encountered in standard Galerkin formulations of the Helmholtz equation.
    \item \textbf{Mesh Flexibility:} The method inherits the WG advantage of applicability to arbitrary polytopal meshes with hanging nodes, requiring neither matching grids nor $C^0$-continuous basis functions.
\end{itemize}

In this work, we provide a rigorous theoretical foundation for the LS-WG scheme, establish the uniqueness of the discrete solution, and derive optimal-order error estimates in a discrete energy norm. Extensive numerical experiments validate these findings and demonstrate the method's robustness in reconstructing solutions to severely ill-posed problems.

The remainder of this paper is organized as follows. Section 2 reviews the mathematical definition of the weak Laplacian. Section 3 introduces the LS-WG formulation for the Helmholtz Cauchy problem. Section 4 and 5 establish the uniqueness of the numerical solution and derive the error equations. Section 6 provides the proof of optimal-order error estimates. Finally, Section 7 presents numerical experiments demonstrating the stability and efficiency of the proposed framework.

Throughout this work, we adopt standard notation for Sobolev spaces and their associated norms. For any open, bounded domain $D \subset \mathbb{R}^d$ with a Lipschitz continuous boundary, $\|\cdot\|_{s,D}$ and $|\cdot|_{s,D}$ denote the norm and seminorm of the Sobolev space $H^s(D)$ for $s \ge 0$, respectively. The corresponding inner product is denoted by $(\cdot, \cdot)_{s,D}$. In the special case where $s=0$, the space $H^0(D)$ coincides with $L^2(D)$, with the norm and inner product denoted by $\|\cdot\|_D$ and $(\cdot, \cdot)_D$, respectively. For simplicity, the subscript $D$ is omitted when $D = \Omega$ or when the domain of integration is clear from context.

\section{Discrete Weak Laplacian}

In this section, we present the formal definitions of the weak Laplacian operator and its discrete counterpart.  

Let $T$ be a polygonal domain in $\mathbb{R}^2$ or a polyhedral domain in $\mathbb{R}^3$ with boundary $\partial T$. We define a \emph{weak function} on $T$ as an ordered triplet $v=\{v_0, v_b, \mathbf{v}_g\}$, where the components are defined as:
\begin{itemize}
    \item $v_0 \in L^2(T)$: represents the value of $v$ in the interior of $T$;
    \item $v_b \in L^2(\partial T)$: represents the value of $v$ on the boundary $\partial T$;
    \item $\mathbf{v}_g \in [L^2(\partial T)]^d$: represents the values of the gradient $\nabla v$ on the boundary $\partial T$.
\end{itemize}

A defining feature of the WG framework is that $v_b$ and $\mathbf{v}_g$ are defined independently and are not required to be the traces of $v_0$ or $\nabla v_0$ on $\partial T$, although such a choice is admissible. We denote the space of all weak functions on $T$ by
\begin{equation*} 
\mathcal{W}(T) = \left\{ v = \{v_0, v_b, \mathbf{v}_g\} : v_0 \in L^2(T), \ v_b \in L^2(\partial T), \ \mathbf{v}_g \in [L^2(\partial T)]^d \right\}.
\end{equation*}

The \emph{weak Laplacian} of a function $v \in \mathcal{W}(T)$, denoted by $\Delta_w v$, is defined as a linear functional in the dual space of $H^2(T)$ such that
\begin{equation*} 
(\Delta_w v, w)_T = (v_0, \Delta w)_T - \langle v_b, \nabla w \cdot \mathbf{n} \rangle_{\partial T} + \langle \mathbf{v}_g \cdot \mathbf{n}, w \rangle_{\partial T}, \quad \forall w \in H^2(T),
\end{equation*}
where $\mathbf{n}$ denotes the outward unit normal vector on $\partial T$, $(\cdot, \cdot)_T$ is the $L^2(T)$ inner product, and $\langle \cdot, \cdot \rangle_{\partial T}$ is the $L^2(\partial T)$ inner product.

Let $P_r(T)$ be the space of polynomials of degree at most $r$ on $T$. The \emph{discrete weak Laplacian} of $v \in \mathcal{W}(T)$, denoted by $\Delta_{w,r,T} v$, is defined as the unique polynomial in $P_r(T)$ satisfying 
\begin{equation}\label{dislap}
(\Delta_{w,r,T} v, w)_T = (v_0, \Delta w)_T - \langle v_b, \nabla w \cdot \mathbf{n} \rangle_{\partial T} + \langle \mathbf{v}_g \cdot \mathbf{n}, w \rangle_{\partial T},  
\end{equation}
 for all $w \in P_r(T)$.
If the interior component $v_0$ possesses $H^2$ regularity (i.e., $v_0 \in H^2(T)$), the discrete weak Laplacian can be equivalently expressed through integration by parts as:
\begin{equation}\label{dislap2}
(\Delta_{w,r,T} v, w)_T = (\Delta v_0, w)_T + \langle v_0 - v_b, \nabla w \cdot \mathbf{n} \rangle_{\partial T} + \langle (\mathbf{v}_g - \nabla v_0) \cdot \mathbf{n}, w \rangle_{\partial T},
\end{equation}
 for all $w \in P_r(T)$. 
 
This representation explicitly highlights the role of the boundary components $v_b$ and $\mathbf{v}_g$ in penalizing the jump between the interior and boundary values.
\section{Least-Squares Weak Galerkin Algorithm}\label{Section:WGFEM}

In this section, we introduce the least-squares weak Galerkin (LS-WG) finite element discretization for the Helmholtz Cauchy problem \eqref{model}. 

Let $\mathcal{T}_h$ be a shape-regular partition of the domain $\Omega$ into polygonal (2D) or polyhedral (3D) elements, following the definitions in \cite{wy3655}. We denote the set of all edges (or faces in 3D) in $\mathcal{T}_h$ by $\mathcal{E}_h$, and let $\mathcal{E}_h^0 = \mathcal{E}_h \setminus \partial\Omega$ represent the set of all interior edges. For each element $T \in \mathcal{T}_h$, $h_T$ denotes the diameter of $T$, and the global mesh size is defined as $h = \max_{T \in \mathcal{T}_h} h_T$.

For a given integer $m \ge 1$, we define the local weak finite element space $W_m(T)$ as:
\[
W_m(T) = \bigl\{ v = \{v_0, v_b, \mathbf{v}_g\} : v_0 \in P_m(T), \ v_b \in P_m(e), \ \mathbf{v}_g \in [P_{m-1}(e)]^d, \ e \subset \partial T \bigr\}.
\]
The global weak finite element space $W_h$ is constructed by patching the local spaces $W_m(T)$ across interior edges such that the boundary trace $v_b$ is single-valued on each $e \in \mathcal{E}_h^0$. We then define $W_h^0 \subset W_h$ as the subspace of weak functions with vanishing Cauchy data on the boundary $\Gamma_1$:
\[
W_h^0 = \bigl\{ v = \{v_0, v_b, \mathbf{v}_g\} \in W_h : v_b|_e = 0, \ \mathbf{v}_g \cdot \mathbf{n}|_e = 0, \ e \subset \Gamma_1 \bigr\}.
\]

For any $v \in W_h$, let $\Delta_w v$ denote the discrete weak Laplacian, computed element-wise as:
\[
(\Delta_w v)|_T = \Delta_{w,r,T}(v|_T), \quad \forall T \in \mathcal{T}_h.
\]

To enforce continuity between the interior and boundary components of the weak functions, we introduce the following stabilizer $s(\cdot, \cdot)$:
\begin{equation*}\label{stabilizer}
s(u, v) = \sum_{T \in \mathcal{T}_h}  k^2 h_T^{-3} \langle u_0 - u_b, v_0 - v_b \rangle_{\partial T} + k^2 h_T^{-1} \langle (\nabla u_0 - \mathbf{u}_g) \cdot \mathbf{n}, (\nabla v_0 - \mathbf{v}_g) \cdot \mathbf{n} \rangle_{\partial T},  
\end{equation*}
for all $u, v \in W_h$. 
 
The least-squares bilinear form $a(\cdot, \cdot)$ for the LS-WG method is then defined as:
\begin{equation}\label{bilinear_a}
a(u, v) = \sum_{T \in \mathcal{T}_h} \left( \Delta_w u + k^2u_0, \Delta_w v + k^2v_0 \right)_T + s(u, v), \quad \forall u, v \in W_h.
\end{equation}

Next, we define the necessary $L^2$ projection operators. For each $T \in \mathcal{T}_h$ and its corresponding edges $e \subset \partial T$, let:
\begin{itemize}
    \item $Q_0: L^2(T) \to P_m(T)$ be the local projection onto the interior polynomial space;
    \item $Q_b: L^2(e) \to P_m(e)$ be the local projection onto the boundary polynomial space;
    \item $Q_n: L^2(e) \to P_{m-1}(e)$ be the local projection for the normal component of the gradient;
    \item $\mathbf{Q}_g: [L^2(e)]^d \to [P_{m-1}(e)]^d$ be the local vector projection for the gradient component.
\end{itemize}
For the exact solution $u$ of \eqref{model}, its projection into the global space $W_h$ is denoted by $Q_h u = \{Q_0 u, Q_b u, \mathbf{Q}_g (\nabla u)\}$.

The LS-WG finite element scheme for the Cauchy problem \eqref{model} is formulated as follows:

\begin{center}
\fbox{
\begin{minipage}{0.9\textwidth}
\textbf{Algorithm 3.1 (LS-WG Scheme).} 
Find $u_h = \{u_0, u_b, \mathbf{u}_g\} \in W_h$ satisfying the boundary conditions $u_b = Q_b g_1$ and $\mathbf{u}_g \cdot \mathbf{n} = Q_n g_2$ on $\Gamma_1$, such that
\begin{equation}\label{al-general}
a(u_h, v_h) = \sum_{T \in \mathcal{T}_h} (f, \Delta_w v_h + k^2v_0)_T, \quad \forall v_h \in W_h^0.
\end{equation}
\end{minipage}
}
\end{center}
 
This least-squares formulation naturally transforms the inherently ill-posed Cauchy problem for the indefinite Helmholtz operator into a symmetric and positive-definite (SPD) algebraic system. Furthermore, the inclusion of the stabilizer $s(u, v)$ ensures that the method remains well-posed and that the interior and boundary components converge at optimal rates.

\section{Solution Uniqueness}\label{Section:EU}

In this section, we establish the existence and uniqueness of the solution to the LS-WG scheme \eqref{al-general}. We begin by identifying a key commutative property of the $L^2$ projection operators, which is essential for the subsequent analysis.

Recall $Q_0$ denote the locally defined $L^2$ projection onto the polynomial space $P_{m}(T)$  for each element $T \in \mathcal{T}_h$.

\begin{lemma}[Commutative Property]\label{Lemma:commute}
The $L^2$ projection operators $Q_h$ and $Q_0$ satisfy the following commutative property:
\begin{equation}\label{eq:commute2}
\Delta_w(Q_h w) = Q_0 (\Delta w), \quad \forall w \in H^2(\Omega).
\end{equation}
\end{lemma}

\begin{proof}
For any $q \in P_{m}(T)$, the definition of the discrete weak Laplacian \eqref{dislap} gives:
\[
\begin{aligned}
(\Delta_w Q_h w, q)_T &= (Q_0 w, \Delta q)_T - \langle Q_b w, \nabla q \cdot \mathbf{n} \rangle_{\partial T} + \langle \mathbf{Q}_g(\nabla w) \cdot \mathbf{n}, q \rangle_{\partial T} \\
&= (w, \Delta q)_T - \langle w, \nabla q \cdot \mathbf{n} \rangle_{\partial T} + \langle \nabla w \cdot \mathbf{n}, q \rangle_{\partial T} \\
&= (\Delta w, q)_T = (Q_0 \Delta w, q)_T,
\end{aligned}
\]
where we have used the facts that $\Delta q \in P_{m-2}(T) \subset P_m(T)$, $\nabla q \cdot \mathbf{n} \in P_{m-1}(e) \subset P_m(e)$, and $q \in P_{m}(T)$. This proves \eqref{eq:commute2}.
\end{proof}

\begin{lemma}[Uniqueness]\label{thmunique1}
Assume that the continuous Cauchy problem \eqref{model} admits a unique solution. Then, the LS-WG finite element scheme \eqref{al-general} possesses a unique solution $u_h \in W_h$.
\end{lemma}

\begin{proof}
It suffices to show that the homogeneous problem ($f=0$, $g_1=0$, $g_2=0$) admits only the trivial solution $u_h = 0$. Let $u_h \in W_h^0$ satisfy:
\[
a(u_h, v_h) = 0, \quad \forall v_h \in W_h^0.
\]
By choosing $v_h = u_h$, we obtain $a(u_h, u_h) = 0$. Based on the definition of the bilinear form \eqref{bilinear_a}, this implies:
\begin{enumerate}
    \item $s(u_h, u_h) = 0$, which yields $u_0 = u_b$ and $(\nabla u_0 - \mathbf{u}_g) \cdot \mathbf{n} = 0$ on $\partial T$ for all $T \in \mathcal{T}_h$.
    \item $\Delta_w u_h + k^2 u_0 = 0$ on each $T \in \mathcal{T}_h$.
\end{enumerate}

From the property $u_0 = u_b$ on $\partial T$ and the fact that $u_b$ is single-valued on the interior edges $\mathcal{E}_h^0$, we conclude that $u_0$ is continuous across all element interfaces, meaning $u_0 \in H^1(\Omega)$. Furthermore, since $(\nabla u_0 - \mathbf{u}_g) \cdot \mathbf{n} = 0$ and $\mathbf{u}_g \cdot \mathbf{n}$ is single-valued on interior edges, the normal component of the gradient $\nabla u_0 \cdot \mathbf{n}$ is also continuous across interfaces. This implies $u_0 \in H^2(\Omega)$.

Using \eqref{dislap2} with $u_0 = u_b$ and $\nabla u_0 \cdot \mathbf{n} = \mathbf{u}_g \cdot \mathbf{n}$, we find that on each element $T$:
\[
\Delta_w u_h = \Delta u_0.
\]
Substituting this into the residual equation yields:
\[
\Delta u_0 + k^2 u_0 = 0 \quad \text{in } \Omega.
\]
Finally, the boundary conditions for $u_h \in W_h^0$ imply $u_b = 0$ and $\mathbf{u}_g \cdot \mathbf{n} = 0$ on $\Gamma_1$, which translates to $u_0 = 0$ and $\nabla u_0 \cdot \mathbf{n} = 0$ on $\Gamma_1$. By the uniqueness assumption for the continuous Cauchy problem, it follows that $u_0 \equiv 0$ in $\Omega$. This in turn forces $u_b = 0$ and $\mathbf{u}_g = \mathbf{0}$. Thus, $u_h \equiv 0$, completing the proof.
\end{proof}

We define the energy norm on the finite element space $W_h^0$ by:
\begin{equation*}
\3bar v \3bar := \sqrt{a(v, v)}.
\end{equation*}
Following the logic established in Lemma \ref{thmunique1}, it is evident that $\3bar \cdot \3bar$ satisfies the properties of a norm on $W_h^0$, providing a robust framework for the subsequent error analysis.
\section{Error Equations}\label{Section:ErrorEquation}

In this section, we derive the error equation that governs the relationship between the exact solution $u$ and its LS-WG approximation $u_h$. Let $u$ be the exact solution of \eqref{model}, and let $u_h \in W_h$ be the solution to the discrete problem \eqref{al-general}. We define the discrete error as
\begin{equation*}\label{error}
e_h := u_h - Q_h u = \{u_0 - Q_0 u, \ u_b - Q_b u, \ \mathbf{u}_g - \mathbf{Q}_g \nabla u\}.
\end{equation*}

\begin{lemma}[Error Equation]\label{errorequa}
For any test function $v_h \in W_h^0$, the error function $e_h$ satisfies the following identity:
\begin{equation}\label{erroreqn}
a(e_h, v_h) = -s(Q_h u, v_h).
\end{equation}
\end{lemma}

\begin{proof}
Let $\mathcal{L}$ denote the continuous operator defined by $\mathcal{L} u := \Delta u + k^2 u = f$. Testing this equation with the discrete operator $(\mathcal{L}_w v_h) |_T := \Delta_w v_h + k^2 v_0$ on each element $T \in \mathcal{T}_h$, we obtain
\begin{equation*}\label{eq:err_proof1}
\sum_{T \in \mathcal{T}_h} ( \Delta u + k^2 u, \Delta_w v_h + k^2 v_0)_T = \sum_{T \in \mathcal{T}_h}(f,  \Delta_w v_h + k^2 v_0)_T.
\end{equation*}
Note that $(\Delta_w v_h + k^2 v_0) \in P_{m}(T)$ on each element $T \in \mathcal{T}_h$. By applying the commutative property \eqref{eq:commute2} established in Lemma \ref{Lemma:commute}, specifically $Q_0(\Delta u) = \Delta_w Q_h u$, we have
\begin{equation}\label{eq:err_proof2}
\sum_{T \in \mathcal{T}_h} ( \Delta_w Q_h u + k^2 Q_0 u,  \Delta_w v_h + k^2 v_0)_T = \sum_{T \in \mathcal{T}_h} (f,  \Delta_w v_h + k^2 v_0)_T.
\end{equation}
Recalling the definition of the least-squares bilinear form $a(\cdot, \cdot)$ in \eqref{bilinear_a}, the left-hand side of \eqref{eq:err_proof2} can be rewritten as $a(Q_h u, v_h) - s(Q_h u, v_h)$. Thus,
\begin{equation}\label{eq:err_proof3}
a(Q_h u, v_h) - s(Q_h u, v_h) = \sum_{T \in \mathcal{T}_h} (f,  \Delta_w v_h + k^2 v_0)_T.
\end{equation}
Subtracting \eqref{eq:err_proof3} from the discrete LS-WG scheme \eqref{al-general} yields
\[
a(u_h, v_h) - \left(a(Q_h u, v_h) - s(Q_h u, v_h)\right) = 0.
\]
By the linearity of $a(\cdot, \cdot)$, we obtain
\[
a(u_h - Q_h u, v_h) + s(Q_h u, v_h) = 0,
\]
which simplifies to the desired error equation $a(e_h, v_h) = -s(Q_h u, v_h)$.
\end{proof}
\section{Error Estimates}\label{Section:ErrorEstimates}

In this section, we establish optimal-order error estimates for the LS-WG approximation in the energy norm. Throughout the following analysis, we denote by $C$ a generic positive constant that is independent of the mesh parameter $h$ and the wavenumber $k$.

\newcommand{\triplenorm}[1]{{\left\vert\kern-0.25ex\left\vert\kern-0.25ex\left\vert #1 \right\vert\kern-0.25ex\right\vert\kern-0.25ex\right\vert}}
For any $\phi \in H^1(T)$, the following trace inequality holds:
\begin{align}
\|\phi\|_{\partial T}^2 &\le C (h_T^{-1} \|\phi\|_T^2 + h_T \|\nabla \phi\|_T^2). \label{trace_H1} 
\end{align}
\begin{lemma}[Approximation Property]\label{lem:approx}
Let $\mathcal{T}_h$ be a shape-regular finite element partition of $\Omega$. For any $u \in H^{m+1}(\Omega)$, the following approximation estimate holds:
\begin{equation}\label{error_H1}
\sum_{T \in \mathcal{T}_h} \|u - Q_0 u\|_{s,T}^2 \le C h^{2(m+1-s)} \|u\|_{m+1}^2, \quad s=0, 1, 2.
\end{equation}
\end{lemma}

\begin{theorem} \label{thm:convergence}
Let $u \in H^{m+1}(\Omega)$ be the exact solution of the Cauchy Helmholtz problem \eqref{model}, and let $u_h \in W_h$ be the LS-WG solution defined by \eqref{al-general}. Then, there exists a constant $C > 0$ such that
\begin{equation}\label{main_estimate}
\triplenorm{u_h - Q_h u} \le C kh^{m-1} \|u\|_{m+1}.
\end{equation}
\end{theorem}

\begin{proof}
We begin by setting the test function $v_h = e_h$ in the error equation \eqref{erroreqn} derived in Lemma \ref{errorequa}. This yields:
\[
\triplenorm{e_h}^2 = a(e_h, e_h) = -s(Q_h u, e_h).
\]
Applying the Cauchy--Schwarz inequality to the stabilization form $s(\cdot, \cdot)$, we obtain:
\[
\triplenorm{e_h}^2 \le |s(Q_h u, e_h)| \le \sqrt{s(Q_h u, Q_h u)} \sqrt{s(e_h, e_h)} \le \sqrt{s(Q_h u, Q_h u)} \triplenorm{e_h},
\]
which implies $\triplenorm{e_h} \le \sqrt{s(Q_h u, Q_h u)}$. We now proceed to bound the two components of the stabilizer $s(Q_h u, Q_h u)$ individually.
 
  Applying the trace inequality \eqref{trace_H1} and the approximation estimate \eqref{error_H1}, we find:
\[
\begin{aligned}
&\sum_{T \in \mathcal{T}_h}k^2 h_T^{-3} \|Q_0 u - Q_b u\|_{\partial T}^2\\ &\le C \sum_{T \in \mathcal{T}_h} k^2 h_T^{-3} \left( h_T^{-1} \|Q_0 u - u\|_T^2 + h_T \|Q_0 u - u\|_{1,T}^2 \right) \\
&\le C \sum_{T \in \mathcal{T}_h} k^2 \left( h_T^{-4} h_T^{2m+2} + h_T^{-2} h_T^{2m} \right) \|u\|_{m+1,T}^2 \\
&\le C k^2 h^{2m-2} \|u\|_{m+1}^2.
\end{aligned}
\]
 
Similarly, applying the trace inequality \eqref{trace_H1} and the approximation estimate \eqref{error_H1}:
\[
\begin{aligned}
&\sum_{T \in \mathcal{T}_h} k^2 h_T^{-1} \|(\nabla Q_0 u - \mathbf{Q}_g \nabla u) \cdot \mathbf{n}\|_{\partial T}^2\\ &\le \sum_{T \in \mathcal{T}_h} k^2 h_T^{-1} \|\nabla Q_0 u - \nabla u\|_{\partial T}^2 \\
&\le C \sum_{T \in \mathcal{T}_h} k^2 h_T^{-1} \left( h_T^{-1} \|\nabla Q_0 u - \nabla u\|_{0,T}^2 + h_T \|\nabla Q_0 u - \nabla u\|_{1,T}^2 \right) \\
&\le C \sum_{T \in \mathcal{T}_h} k^2 (h_T^{-2} h_T^{2m} + h_T^{2m-2}) \|u\|_{m+1,T}^2 \\
&\le C k^2 h^{2m-2} \|u\|_{m+1}^2.
\end{aligned}
\]
Combining the estimates, we conclude $\triplenorm{e_h}^2 \le  C k^2 h^{2m-2} \|u\|_{m+1}^2$, which yields the desired estimate \eqref{main_estimate}.
\end{proof}

\section{Numerical experiment}

In the first numerical test,  we solve the Cauchy problem for the Helmholtz equation \eqref{model}, where 
\an{\label{data}  \Omega=(0,1)\times(0,1), \ \ \Gamma_1=\partial \Omega \setminus \{0\}\times (0,1), \ \ k^2=10 \ \t{or} \ 10^{6} ,  }
and  $(f, g_1, g_2)$ are chosen so that the exact solution is smooth, 
\an{\label{s2}
   u =-(2x^3 + y + 1)^2.  }  
 The solution \eqref{s2} is computed on the triangular grids shown in Figure \ref{f-2}, and on the non-convex polygonal 
  grids shown in Figure \ref{f-5}, by 
  the weak Galerkin $P_k$-$P_k$-$P_{k-1}^2$/$P_k^{2\times 2}$ finite elements, $k=2,3$ and $4$.
The results are listed in Tables \ref{t1}-\ref{t3}, 
    where we can see that the optimal orders of convergence 
  are achieved roughly.  
  In these tables, $G_i$ denotes the $i$-th grid, e.g., $G_3$ in Table \ref{t2} is shown in
    Figure \ref{f-2} or Figure \ref{f-5} .

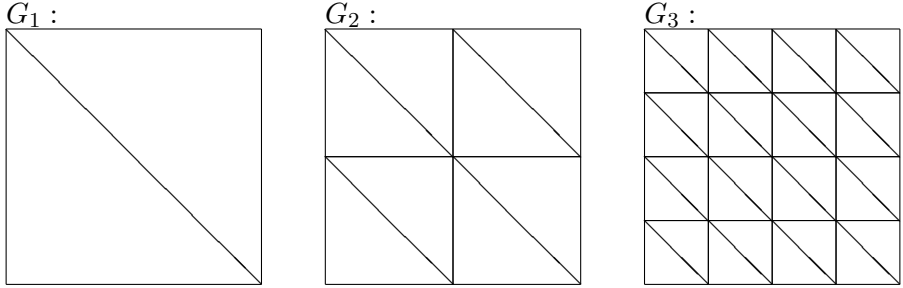
\begin{figure}[H]
\begin{center}\setlength\unitlength{2.4pt}\centering 
 \begin{picture}(140,45)(0,0) \put(0,41){$G_1:$}  \put(50,41){$G_2:$} \put(100,41){$G_3:$} 
  
\def\sq{\begin{picture}(40,40)(0,0) \put(0,40){\line(1,-1){40}}
  \multiput(0,0)(40,0){2}{\line(0,1){40}}\multiput(0,0)(0,40){2}{\line(1,0){40}} \end{picture} }
  
\put(0,0){\begin{picture}(40,40)(0,0)
  \multiput(0,0)(0,40){1}{\multiput(0,0)(40,0){1}{\sq}} 
  \end{picture} }
  
\put(50,0){\setlength\unitlength{1.2pt}\begin{picture}(40,40)(0,0)
  \multiput(0,0)(0,40){2}{\multiput(0,0)(40,0){2}{\sq}} 
  \end{picture} } 
\put(100,0){\setlength\unitlength{0.6pt}\begin{picture}(40,40)(0,0)
  \multiput(0,0)(0,40){4}{\multiput(0,0)(40,0){4}{\sq}} 
  \end{picture} } 
\end{picture}\end{center}
\caption{The triangular  grids used in Tables \ref{t1}--\ref{t9}. }
\label{f-2}
\end{figure}

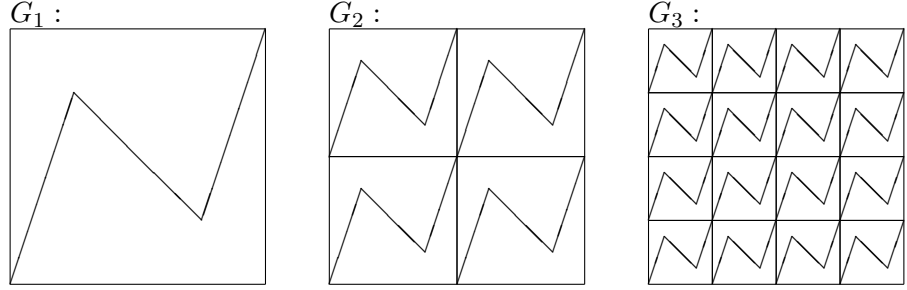
\begin{figure}[H]
\begin{center}\setlength\unitlength{2.4pt}\centering 
 \begin{picture}(140,45)(0,0) \put(0,41){$G_1:$}  \put(50,41){$G_2:$} \put(100,41){$G_3:$} 
  
\def\sq{\begin{picture}(40,40)(0,0) \put(0,0){\line(1,3){10}}  \put(40,40){\line(-1,-3){10}} \put(10,30){\line(1,-1){20}}
  \multiput(0,0)(40,0){2}{\line(0,1){40}}\multiput(0,0)(0,40){2}{\line(1,0){40}} \end{picture} }
  
\put(0,0){\begin{picture}(40,40)(0,0)
  \multiput(0,0)(0,40){1}{\multiput(0,0)(40,0){1}{\sq}} 
  \end{picture} }
  
\put(50,0){\setlength\unitlength{1.2pt}\begin{picture}(40,40)(0,0)
  \multiput(0,0)(0,40){2}{\multiput(0,0)(40,0){2}{\sq}} 
  \end{picture} } 

\put(100,0){\setlength\unitlength{0.6pt}\begin{picture}(40,40)(0,0)
  \multiput(0,0)(0,40){4}{\multiput(0,0)(40,0){4}{\sq}} 
  \end{picture} } 
\end{picture}\end{center}
\caption{The non-convex polygonal grids used in Tables \ref{t1}--\ref{t9}. }
\label{f-5}
\end{figure}

\begin{table}[H]
  \centering  \renewcommand{\arraystretch}{1.1}
  \caption{Error profile by the $P_2$ WG element  for computing \eqref{s2}. }
  \label{t1}
\begin{tabular}{c|cc|cc}
\hline
Grid $G_i$ & \quad $\| u-u_h\|_{0}$ & $O(h^r)$ & \  $\|\Delta_w( u-u_h)\|_0 $& $O(h^r)$   \\ \hline
    &  \multicolumn{4}{c}{On triangular meshes (Figure \ref{f-2}), $k^2=10$}    \\
\hline  
 4&    0.752E-02 &  2.8&    0.314E+01 &  1.3\\
 5&    0.121E-02 &  2.6&    0.147E+01 &  1.1\\
 6&    0.248E-03 &  2.3&    0.726E+00 &  1.0\\
\hline 
    &  \multicolumn{4}{c}{On triangular meshes (Figure \ref{f-2}),  $k^2=10^{6}$}    \\
\hline  
 4&    0.424E-03 &  0.8&    0.799E+01 &  0.9\\
 5&    0.209E-03 &  1.0&    0.344E+01 &  1.2\\
 6&    0.375E-04 &  2.5&    0.182E+01 &  0.9\\
\hline 
    &  \multicolumn{4}{c}{ On  polygonal meshes (Figure \ref{f-5}),  $k^2=10$}    \\
\hline   
 4&    0.102E-01 &  3.0&    0.133E+02 &  1.0\\
 5&    0.117E-02 &  3.1&    0.660E+01 &  1.0\\
 6&    0.135E-03 &  3.1&    0.329E+01 &  1.0\\
\hline 
    &  \multicolumn{4}{c}{ On  polygonal meshes (Figure \ref{f-5}),  $k^2=10^{6}$}    \\
\hline   
 4&    0.223E-02 &  0.5&    0.213E+02 &  1.1\\
 5&    0.731E-03 &  1.6&    0.772E+01 &  1.5\\
 6&    0.116E-03 &  2.7&    0.389E+01 &  1.0\\
\hline 
    \end{tabular}%
\end{table}%

\begin{table}[H]
  \centering  \renewcommand{\arraystretch}{1.1}
  \caption{Error profile by the $P_3$ WG element  for computing \eqref{s2}. }
  \label{t2}
\begin{tabular}{c|cc|cc}
\hline
Grid $G_i$ & \quad $\| u-u_h\|_{0}$ & $O(h^r)$ & \  $\|\Delta_w( u-u_h)\|_0 $& $O(h^r)$   \\ \hline
    &  \multicolumn{4}{c}{On triangular meshes (Figure \ref{f-2}), $k^2=10$}    \\
\hline  
 3&    0.142E-02 &  4.2&    0.765E+00 &  2.1\\
 4&    0.659E-04 &  4.4&    0.184E+00 &  2.1\\
 5&    0.354E-05 &  4.2&    0.463E-01 &  2.0\\
\hline 
    &  \multicolumn{4}{c}{On triangular meshes (Figure \ref{f-2}),  $k^2=10^{6}$}    \\
\hline  
 3&    0.151E-03 &  2.2&    0.116E+01 &  2.1\\
 4&    0.254E-04 &  2.6&    0.191E+00 &  2.6\\
 5&    0.347E-05 &  2.9&    0.374E-01 &  2.4\\
\hline 
    &  \multicolumn{4}{c}{ On  polygonal meshes (Figure \ref{f-5}),  $k^2=10$}    \\
\hline   
 3&    0.555E-02 &  3.9&    0.857E+01 &  1.9\\
 4&    0.347E-03 &  4.0&    0.217E+01 &  2.0\\
 5&    0.507E-04 &  2.8&    0.547E+00 &  2.0\\
\hline 
    &  \multicolumn{4}{c}{ On  polygonal meshes (Figure \ref{f-5}),  $k^2=10^{6}$}    \\
\hline   
 3&    0.727E-03 &  2.3&    0.107E+02 &  2.0\\
 4&    0.122E-03 &  2.6&    0.257E+01 &  2.1\\
 5&    0.139E-04 &  3.1&    0.642E+00 &  2.0\\
\hline 
    \end{tabular}%
\end{table}%

\begin{table}[H]
  \centering  \renewcommand{\arraystretch}{1.1}
  \caption{Error profile by the $P_4$ WG element  for computing \eqref{s2}. }
  \label{t3}
\begin{tabular}{c|cc|cc}
\hline
Grid $G_i$ & \quad $\| u-u_h\|_{0}$ & $O(h^r)$ & \  $\|\Delta_w( u-u_h)\|_0 $& $O(h^r)$   \\ \hline
    &  \multicolumn{4}{c}{On triangular meshes (Figure \ref{f-2}), $k^2=10$}    \\
\hline  
 3&    0.281E-01 &  4.8&    0.530E+02 &  2.8\\
 4&    0.713E-02 &  2.0&    0.110E+02 &  2.3\\
 5&    0.387E-03 &  4.2&    0.143E+01 &  2.9\\
\hline 
    &  \multicolumn{4}{c}{On triangular meshes (Figure \ref{f-2}),  $k^2=10^{6}$}    \\
\hline  
 3&    0.161E-02 &  2.8&    0.776E+02 &  2.8\\
 4&    0.223E-03 &  2.8&    0.216E+02 &  1.8\\
 5&    0.208E-04 &  3.4&    0.191E+01 &  3.5\\
\hline 
    &  \multicolumn{4}{c}{ On  polygonal meshes (Figure \ref{f-5}),  $k^2=10$}    \\
\hline   
 3&    0.382E-03 &  5.0&    0.180E+01 &  3.0\\
 4&    0.223E-03 &  0.8&    0.225E+00 &  3.0\\
 5&    0.397E-01 &  ---&    0.447E+00 &  ---\\
\hline 
    &  \multicolumn{4}{c}{ On  polygonal meshes (Figure \ref{f-5}),  $k^2=10^{6}$}    \\
\hline   
 3&    0.543E-04 &  3.9&    0.174E+01 &  3.0\\
 4&    0.557E-05 &  3.3&    0.185E+00 &  3.2\\
 5&    0.268E-05 &  1.1&    0.215E+00 &  ---\\
\hline 
    \end{tabular}%
\end{table}%

In the second numerical test,  we solve a Cauchy problem \eqref{model} with the data \eqref{data} and an oscillating exact solution 
\an{\label{s3}
   u =\sin(4\pi x) \sin(4 \pi y).  }   
The solution makes the Cauchy problem more ill.  
The computational results are thus worse than those for the first solution \eqref{s2}, when $k^2$ is small.
Again, we compute \eqref{s3} on the triangular grids (Figure \ref{f-2}), and on the non-convex polygonal 
  grids(Figure \ref{f-5}), by 
  the weak Galerkin $P_k$-$P_k$-$P_{k-1}^2$/$P_k^{2\times 2}$ finite elements, $k=2,3$ and $4$.
The results are listed in Tables \ref{t4}-\ref{t6}.

\begin{table}[H]
  \centering  \renewcommand{\arraystretch}{1.1}
  \caption{Error profile by the $P_2$ WG element  for computing \eqref{s3}. }
  \label{t4}
\begin{tabular}{c|cc|cc}
\hline
Grid $G_i$ & \quad $\| u-u_h\|_{0}$ & $O(h^r)$ & \  $\|\Delta_w( u-u_h)\|_0 $& $O(h^r)$   \\ \hline
    &  \multicolumn{4}{c}{On triangular meshes (Figure \ref{f-2}), $k^2=10$}    \\
\hline  
 4&    0.148E+00 &  0.7&    0.406E+02 &  1.2\\
 5&    0.712E-01 &  1.1&    0.202E+02 &  1.0\\
 6&    0.377E-01 &  0.9&    0.103E+02 &  1.0\\
\hline 
    &  \multicolumn{4}{c}{On triangular meshes (Figure \ref{f-2}),  $k^2=10^{6}$}    \\
\hline  
 4&    0.718E-02 &  1.1&    0.121E+03 &  0.4\\
 5&    0.289E-02 &  1.3&    0.537E+02 &  1.2\\
 6&    0.594E-03 &  2.3&    0.332E+02 &  0.7\\
\hline 
    &  \multicolumn{4}{c}{ On  polygonal meshes (Figure \ref{f-5}),  $k^2=10$}    \\
\hline   
 4&    0.224E+00 &  1.6&    0.246E+03 &  0.5\\
 5&    0.976E-01 &  1.2&    0.128E+03 &  0.9\\
 6&    0.613E-01 &  0.7&    0.643E+02 &  1.0\\
\hline 
    &  \multicolumn{4}{c}{ On  polygonal meshes (Figure \ref{f-5}),  $k^2=10^{6}$}    \\
\hline   
 4&    0.409E-01 &  0.0&    0.334E+03 &  0.3\\
 5&    0.126E-01 &  1.7&    0.152E+03 &  1.1\\
 6&    0.212E-02 &  2.6&    0.711E+02 &  1.1\\
\hline 
    \end{tabular}%
\end{table}%

\begin{table}[H]
  \centering  \renewcommand{\arraystretch}{1.1}
  \caption{Error profile by the $P_3$ WG element  for computing \eqref{s3}. }
  \label{t5}
\begin{tabular}{c|cc|cc}
\hline
Grid $G_i$ & \quad $\| u-u_h\|_{0}$ & $O(h^r)$ & \  $\|\Delta_w( u-u_h)\|_0 $& $O(h^r)$   \\ \hline
    &  \multicolumn{4}{c}{On triangular meshes (Figure \ref{f-2}), $k^2=10$}    \\
\hline  
 3&    0.448E+00 &  0.8&    0.114E+03 &  0.7\\
 4&    0.743E-01 &  2.6&    0.254E+02 &  2.2\\
 5&    0.742E-02 &  3.3&    0.657E+01 &  2.0\\
\hline 
    &  \multicolumn{4}{c}{On triangular meshes (Figure \ref{f-2}),  $k^2=10^{6}$}    \\
\hline  
 3&    0.147E-01 &  2.0&    0.249E+03 &  0.3\\
 4&    0.291E-02 &  2.3&    0.444E+02 &  2.5\\
 5&    0.407E-03 &  2.8&    0.985E+01 &  2.2\\
\hline 
    &  \multicolumn{4}{c}{ On  polygonal meshes (Figure \ref{f-5}),  $k^2=10$}    \\
\hline   
 3&    0.109E+01 &  3.7&    0.124E+04 &  0.7\\
 4&    0.917E-01 &  3.6&    0.395E+03 &  1.6\\
 5&    0.103E-01 &  3.2&    0.103E+03 &  1.9\\
\hline 
    &  \multicolumn{4}{c}{ On  polygonal meshes (Figure \ref{f-5}),  $k^2=10^{6}$}    \\
\hline   
 3&    0.781E-01 &  0.0&    0.167E+04 &  0.1\\
 4&    0.235E-01 &  1.7&    0.344E+03 &  2.3\\
 5&    0.295E-02 &  3.0&    0.104E+03 &  1.7\\
\hline 
    \end{tabular}%
\end{table}%

\begin{table}[H]
  \centering  \renewcommand{\arraystretch}{1.1}
  \caption{Error profile by the $P_4$ WG element  for computing \eqref{s3}. }
  \label{t6}
\begin{tabular}{c|cc|cc}
\hline
Grid $G_i$ & \quad $\| u-u_h\|_{0}$ & $O(h^r)$ & \  $\|\Delta_w( u-u_h)\|_0 $& $O(h^r)$   \\ \hline
    &  \multicolumn{4}{c}{On triangular meshes (Figure \ref{f-2}), $k^2=10$}    \\
\hline  
 3&    0.483E-01 &  4.0&    0.525E+02 &  2.8\\
 4&    0.941E-02 &  2.4&    0.110E+02 &  2.3\\
 5&    0.663E-03 &  3.8&    0.143E+01 &  2.9\\
\hline 
    &  \multicolumn{4}{c}{On triangular meshes (Figure \ref{f-2}),  $k^2=10^{6}$}    \\
\hline  
 3&    0.657E-02 &  3.5&    0.564E+02 &  3.1\\
 4&    0.728E-03 &  3.2&    0.145E+02 &  2.0\\
 5&    0.402E-04 &  4.2&    0.157E+01 &  3.2\\
\hline 
    &  \multicolumn{4}{c}{ On  polygonal meshes (Figure \ref{f-5}),  $k^2=10$}    \\
\hline   
 3&    0.106E+01 &  5.4&    0.390E+04 &  2.3\\
 4&    0.263E-01 &  5.3&    0.469E+03 &  3.1\\
 5&    0.119E-01 &  1.1&    0.607E+02 &  2.9\\
\hline 
    &  \multicolumn{4}{c}{ On  polygonal meshes (Figure \ref{f-5}),  $k^2=10^{6}$}    \\
\hline   
 3&    0.945E-01 &  2.7&    0.303E+04 &  2.8\\
 4&    0.101E-01 &  3.2&    0.424E+03 &  2.8\\
 5&    0.585E-03 &  4.1&    0.592E+02 &  2.8\\
\hline 
    \end{tabular}%
\end{table}%

In the third numerical test,  we solve the Cauchy problem  \eqref{model} with the data \eqref{data} and an internal layer solution, 
\an{\label{s5}
   u =(y^2 - 2 y) (1 + \tanh(20x-10)).  }  
The  $P_4$ WG finite element solution for \eqref{s5} and the errors on the $G_5$ triangular grid are plotted
  in Figure \ref{f-s5}.
 When $k^2$ is small, a large error occurs at the free-boundary.
When $k^2$ is large, the error occurs only at the internal layer.

\begin{figure}[H]
 \begin{center}\setlength\unitlength{1.0pt}
\begin{picture}(400,376)(0,0) 
  \put(0,-125){\includegraphics[width=400pt]{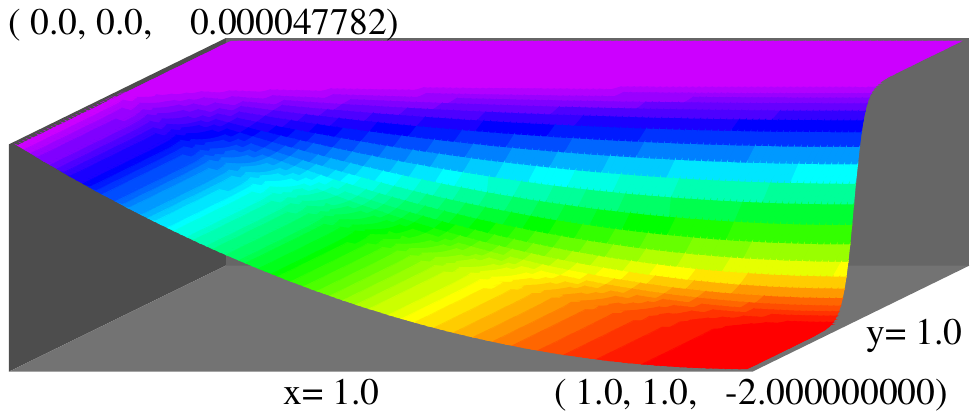}}  
  \put(0,-255){\includegraphics[width=400pt]{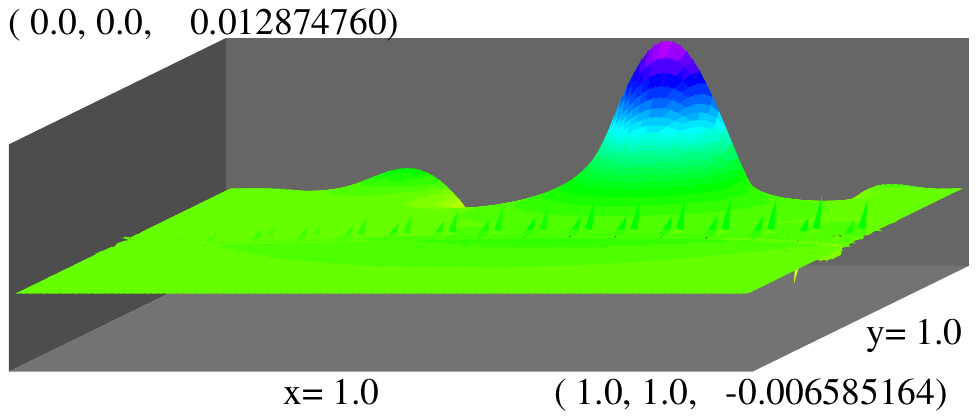}}  
  \put(0,-385){\includegraphics[width=400pt]{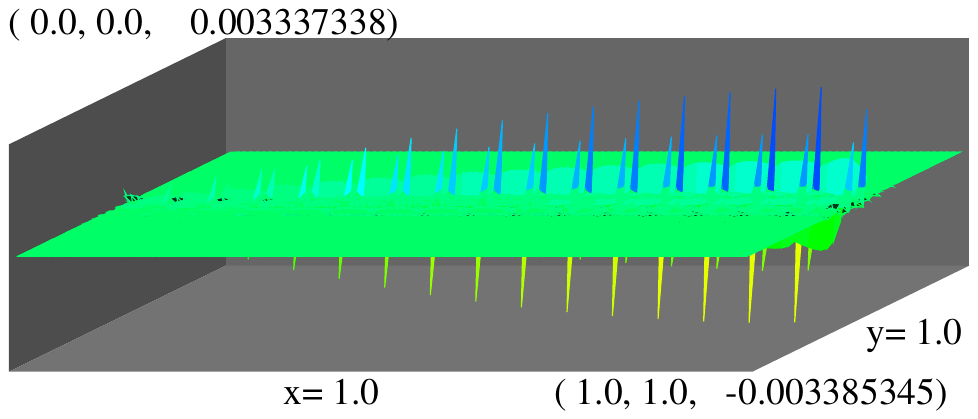}}   
 \end{picture}\end{center}
\caption{The $P_4$ WG solution $u_h$ for \eqref{s5} when $k^2=10^6$ (top), the error $u-u_h$ when $k^2=10$ and when $k^2=10^{6}$ (bottom). }\label{f-s5}
\end{figure}

 We compute the solution \eqref{s5} on the triangular grids shown in Figure \ref{f-2}, and on the non-convex polygonal 
  grids shown in Figure \ref{f-5}, by 
  the weak Galerkin $P_k$-$P_k$-$P_{k-1}^2$/$P_k^{2\times 2}$ finite elements, $k=2,3$ and $4$.
The results are listed in Tables \ref{t7}-\ref{t9}, 
    where   the   order of convergence is not stable, much worse than that of first two examples,
    due to the ill-posed Cauchy problem and the internal layer.

\begin{table}[H]
  \centering  \renewcommand{\arraystretch}{1.1}
  \caption{Error profile by the $P_2$ WG element  for computing \eqref{s5}. }
  \label{t7}
\begin{tabular}{c|cc|cc}
\hline
Grid $G_i$ & \quad $\| u-u_h\|_{0}$ & $O(h^r)$ & \  $\|\Delta_w( u-u_h)\|_0 $& $O(h^r)$   \\ \hline
    &  \multicolumn{4}{c}{On triangular meshes (Figure \ref{f-2}), $k^2=10$}    \\
\hline  
 4&    0.257E-01 &  2.3&    0.190E+02 &  1.2\\
 5&    0.403E-02 &  2.7&    0.129E+02 &  0.6\\
 6&    0.610E-03 &  2.7&    0.927E+01 &  0.5\\
\hline 
    &  \multicolumn{4}{c}{On triangular meshes (Figure \ref{f-2}),  $k^2=10^{6}$}    \\
\hline  
 4&    0.261E-02 &  0.0&    0.368E+02 &  1.1\\
 5&    0.123E-02 &  1.1&    0.287E+02 &  0.4\\
 6&    0.457E-03 &  1.4&    0.210E+02 &  0.5\\
\hline 
    &  \multicolumn{4}{c}{ On  polygonal meshes (Figure \ref{f-5}),  $k^2=10$}    \\
\hline   
 4&    0.178E+00 &  0.9&    0.118E+03 &  0.0\\
 5&    0.164E-01 &  3.4&    0.253E+02 &  2.2\\
 6&    0.998E-03 &  4.0&    0.306E+02 &  0.0\\
\hline 
    &  \multicolumn{4}{c}{ On  polygonal meshes (Figure \ref{f-5}),  $k^2=10^{6}$}    \\
\hline   
 4&    0.113E-01 &  0.0&    0.119E+03 &  0.5\\
 5&    0.220E-02 &  2.4&    0.417E+02 &  1.5\\
 6&    0.972E-03 &  1.2&    0.403E+02 &  0.1\\
\hline 
    \end{tabular}%
\end{table}%

\begin{table}[H]
  \centering  \renewcommand{\arraystretch}{1.1}
  \caption{Error profile by the $P_3$ WG element  for computing \eqref{s5}. }
  \label{t8}
\begin{tabular}{c|cc|cc}
\hline
Grid $G_i$ & \quad $\| u-u_h\|_{0}$ & $O(h^r)$ & \  $\|\Delta_w( u-u_h)\|_0 $& $O(h^r)$   \\ \hline
    &  \multicolumn{4}{c}{On triangular meshes (Figure \ref{f-2}), $k^2=10$}    \\
\hline  
 3&    0.377E-01 &  2.9&    0.391E+02 &  0.3\\
 4&    0.863E-02 &  2.1&    0.190E+02 &  1.0\\
 5&    0.809E-03 &  3.4&    0.101E+02 &  0.9\\
\hline 
    &  \multicolumn{4}{c}{On triangular meshes (Figure \ref{f-2}),  $k^2=10^{6}$}    \\
\hline  
 3&    0.270E-02 &  0.0&    0.314E+02 &  1.6\\
 4&    0.167E-02 &  0.7&    0.514E+02 &  ---\\
 5&    0.655E-03 &  1.4&    0.196E+02 &  1.4\\
\hline 
    &  \multicolumn{4}{c}{ On  polygonal meshes (Figure \ref{f-5}),  $k^2=10$}    \\
\hline   
 3&    0.116E+01 &  0.3&    0.324E+03 &  ---\\
 4&    0.456E+00 &  1.3&    0.107E+03 &  1.6\\
 5&    0.153E-01 &  4.9&    0.101E+03 &  0.1\\
\hline 
    &  \multicolumn{4}{c}{ On  polygonal meshes (Figure \ref{f-5}),  $k^2=10^{6}$}    \\
\hline   
 3&    0.135E-01 &  0.0&    0.322E+03 &  0.0\\
 4&    0.462E-02 &  1.5&    0.999E+02 &  1.7\\
 5&    0.313E-02 &  0.6&    0.973E+02 &  ---\\
\hline 
    \end{tabular}%
\end{table}%

\begin{table}[H]
  \centering  \renewcommand{\arraystretch}{1.1}
  \caption{Error profile by the $P_4$ WG element  for computing \eqref{s5}. }
  \label{t9}
\begin{tabular}{c|cc|cc}
\hline
Grid $G_i$ & \quad $\| u-u_h\|_{0}$ & $O(h^r)$ & \  $\|\Delta_w( u-u_h)\|_0 $& $O(h^r)$   \\ \hline
    &  \multicolumn{4}{c}{On triangular meshes (Figure \ref{f-2}), $k^2=10$}    \\
\hline  
 3&    0.358E-01 &  1.9&    0.241E+02 &  1.6\\
 4&    0.649E-02 &  2.5&    0.228E+02 &  0.1\\
 5&    0.104E-02 &  2.6&    0.299E+01 &  2.9\\
\hline 
    &  \multicolumn{4}{c}{On triangular meshes (Figure \ref{f-2}),  $k^2=10^{6}$}    \\
\hline  
 3&    0.202E-02 &  0.4&    0.743E+02 &  ---\\
 4&    0.118E-02 &  0.8&    0.407E+02 &  0.9\\
 5&    0.877E-04 &  3.7&    0.356E+01 &  3.5\\
\hline 
    &  \multicolumn{4}{c}{ On  polygonal meshes (Figure \ref{f-5}),  $k^2=10$}    \\
\hline   
 3&    0.635E+01 &  0.0&    0.672E+03 &  0.6\\
 4&    0.335E+00 &  4.2&    0.552E+03 &  0.3\\
 5&    0.297E-01 &  3.5&    0.741E+02 &  2.9\\
\hline 
    &  \multicolumn{4}{c}{ On  polygonal meshes (Figure \ref{f-5}),  $k^2=10^{6}$}    \\
\hline   
 3&    0.103E-01 &  1.3&    0.236E+03 &  2.0\\
 4&    0.111E-01 &  ---&    0.477E+03 &  ---\\
 5&    0.735E-03 &  3.9&    0.697E+02 &  2.8\\
\hline 
    \end{tabular}%
\end{table}%

\end{document}